\documentclass[notitlepage]{scrartcl}
\usepackage[shortlabels]{enumitem}
\usepackage{amsmath}
\usepackage{amssymb}
\usepackage{amsthm}
\usepackage{abstract}
\usepackage{xcolor}
\usepackage{comment}
\usepackage{mathtools}
\usepackage{graphicx}
\usepackage{caption}
\usepackage{subcaption}
\usepackage{float}
\usepackage{hyperref}
\usepackage{cleveref}
\usepackage{tikz}
\hypersetup{
    colorlinks=true,
    linkcolor=blue,
    filecolor=magenta,      
    urlcolor=cyan
    }
\makeatletter
\newcommand*{\barfix}[2][.175ex]{%
  \mathpalette{\@barfix{#1}}{#2}%
}
\newcommand*{\@barfix}[3]{%
  \vbox{%
    \kern#1\relax
    \hbox{$#2#3\m@th$}%
  }%
}
\makeatother

\newtheorem{theorem}{Theorem}
\newtheorem{thm}{Theorem}[section]

\newtheorem{lemma}[thm]{Lemma}
\newtheorem{proposition}[thm]{Proposition}
\newtheorem{claim}[thm]{Claim}

\newtheorem{question}[thm]{Question}
\newcommand{\footremember}[2]{%
    \footnote{#2}
    \newcounter{#1}
    \setcounter{#1}{\value{footnote}}%
}

\usepackage[margin=1in]{geometry} 

\providecommand\given{\nonscript\:\ifthenelse{\equal{\delimsize}{}}{\big\vert}{\delimsize\vert}\nonscript\:\mathopen{}}
\let\Pr\undefined
\DeclarePairedDelimiterXPP\Pr[1]{\mathbb{P}}(){}{#1}
\DeclarePairedDelimiterXPP\Ex[1]{\mathbb{E}}{[}{]}{}{#1}

\title{\vspace{-1.5cm}Minimum degree $k$ and $k$-connectedness usually arrive together}
\author{%
Sahar Diskin \footremember{alley}{\scriptsize{School of Mathematical Sciences, Tel Aviv University, Tel Aviv 6997801, Israel. Email: sahardiskin@mail.tau.ac.il.}}%
\and Anna Geisler \footremember{trailer}{\scriptsize{Institute of Discrete Mathematics, Graz University of Technology, Steyrergasse 30, 8010 Graz, Austria. Email: geisler@math.tugraz.at.}}
}
\begin{document}

\maketitle
\begin{abstract}
Let $d,n\in \mathbb{N}$ be such that $d=\omega(1)$, and $d\le n^{1-a}$ for some constant $a>0$. Consider a $d$-regular graph $G=(V, E)$ and the random graph process that starts with the empty graph $G(0)$ and at each step $G(i)$ is obtained from $G(i-1)$ by adding uniformly at random a new edge from $E$. 
We show that if $G$ satisfies some (very) mild global edge-expansion, and an almost optimal edge-expansion of sets up to order $O(d\log n)$, then for any constant $k\in \mathbb{N}$ in the random graph process on $G$, typically the hitting times of minimum degree at least $k$ and of $k$-connectedness are equal. This, in particular, covers both $d$-regular high dimensional product graphs and pseudo-random graphs, and confirms a conjecture of Joos from 2015. We further demonstrate that this result is tight in the sense that there are $d$-regular $n$-vertex graphs with optimal edge-expansion of sets up to order $\Omega(d)$, for which the probability threshold of minimum degree at least one is different than the probability threshold of connectivity.
\end{abstract}

\section{Introduction and main results}
Given a graph $G=(V,E)$, the \textit{random graph process} on $G$ starts with the empty graph $G(0)$ on $V$, and at each step $1\le i\le |E|$, $G(i)$ is obtained from $G(i-1)$ by adding a new edge chosen uniformly at random from $E\setminus E(G(i-1))$. The \textit{hitting time} of a monotone increasing (non-empty) graph property $\mathcal{P}$ is the random variable equal to the minimum index $\tau\coloneqq \tau(\mathcal{P})$ for which $G(\tau)\in \mathcal{P}$, but $G(\tau-1)\notin\mathcal{P}$. Note that when $G=K_n$, then $G(m)$ has the same distribution as $G(n,m)$, the random graph on $n$ vertices where we choose $m$ out of $\binom{n}{2}$ edges uniformly at random. 

In this paper, we provide sufficient conditions on a regular graph $G$, guaranteeing that in the random graph process on $G$ typically the hitting time of having minimum degree at least $k$ is equal to the hitting time of being $k$-connected. We further discuss the tightness of these conditions. Before stating our main result, let us give some motivation and discuss previous results. 

Note that, considering the random graph process on any graph $G$, the hitting time of being $k$-connected is (deterministically) at least the hitting time of having minimum degree $k$. Considering the random graph process on the complete graph, Bollob\'as and Thomason~\cite{BT85} (see also Erd\H{o}s and R\'enyi~\cite{ER66}) showed that the hitting times of minimum degree at least $k$ and $k$-connectedness are the same \textbf{whp}. Bollob\'as~\cite{B90} and Bollob\'as, Kohayakawa, and {\L}uczak~\cite{BKL93} demonstrated that the same holds for the $d$-dimensional hypercube $Q^d$. Joos \cite{J12, J15} further generalised this to Cartesian powers of a given graph (indeed, $Q^d$ is the Cartesian power of $d$ copies of $K_2$). The authors \cite{DG24} showed this for high dimensional Cartesian product graphs where the base graphs are regular and of bounded order for $k=1$ (that is, connectivity and minimum degree one). Moreover, for $p \in (0,1)$ consider the random graph $G_p$ obtained by retaining each edge of $G$ independently with probability $p$. Federico, van der Hofstad, and Hulshof \cite{FHH16} determined the probability threshold of connectivity for the Hamming graph $H(n,d)$ (where the asymptotic are in $n$), which is the Cartesian product of $d$ copies of $K_n$. Gupta, Lee, and Li \cite{GLL21} showed that for $d$-regular and $d$-edge-connected graphs on $n$ vertices when $d=\omega(\sqrt{n})$, the probability threshold for the random subgraph to be connected is asymptotically the same as the probability threshold for minimum degree at least one. 

Throughout the paper, all asymptotics are with $n \to \infty$.
Our first main result gives sufficient conditions on a $d$-regular $n$-vertex graph $G$, guaranteeing that \textbf{whp} the hitting times (in the random graph process on $G$) of minimum degree at least $k$ and $k$-connectedness are the same.
\begin{theorem}\label{th: sufficient}
Let $d,n \in \mathbb{N}$ be such that $d=\omega(1)$ and that there exists some constant $a>0$ such that $d\le n^{1-a}$. Let $\epsilon>0$ be a sufficiently small constant, and let $k\in \mathbb{N}$ be a constant. Then, for every constant $c>0$ there exists a sufficiently large constant $C\coloneqq C(c,k)>0$ such that the following holds. Let $G$ be a $d$-regular graph on $n$ vertices satisfying that
\begin{enumerate}[(P\arabic*{})]
    \item for every $U \subseteq V(G)$ with $|U| \leq n/2$, $e(U, U^C) \geq c |U|$; and, \label{p: global}
    \item for every $U \subseteq V(G)$ with $|U| \le Cd\log n$, $e(U, U^C) \geq \left(1-\epsilon\right) d |U|$. \label{p: local} 
\end{enumerate}
Consider the random graph process on $G$. Let $\tau_k$ be the hitting time of minimum degree at least $k$ and let $\tau_{kc}$ be the hitting time of $k$-connectedness. Then, \textbf{whp}, $\tau_k=\tau_{kc}$.
\end{theorem}
A few comments are in place. In a sense, we show that the phenomenon of typical equality between the hitting time of minimum degree at least $k$ and that of $k$-connectedness can be read off from the edge-expansion properties of the graph: a very mild `global' edge-expansion, and an almost optimal edge-expansion of sets up to order $O(d\log n)$. We note that while we treat $c$ as a constant, we may allow it to tend to zero by taking $C=C(c, k)$ to infinity. Further, we note that Properties \ref{p: global} and \ref{p: local} resemble similar assumptions in \cite{DK2401,DK2402}, under which the uniqueness of the giant component in a general setting was shown (and which were shown to be essentially tight). Finally, we remark that Theorem \ref{th: sufficient} confirms a conjecture of Joos \cite{J15} from 2015.

Observe that by Harper's inequality \cite{H64}, the hypercube $Q^d$ satisfies Properties \ref{p: global} and \ref{p: local}, and since $|V(Q^d)|=2^d$ and $Q^d$ is $d$-regular, it satisfies the degree assumptions of Theorem \ref{th: sufficient}. In fact, any Cartesian product of connected, regular, and bounded order base graphs satisfies the assumptions of Theorem \ref{th: sufficient} --- such graphs have degree of order logarithmic in the number of vertices, and satisfy Harper-like isoperimetric inequalities \cite{DEKKIso23, DS24}. Moreover, both Properties \ref{p: global} and \ref{p: local} can be guaranteed by proving that the second largest eigenvalue $\lambda_2$ of the adjacency matrix of $G$ satisfies $\lambda_2=o(d)$ (see \cite{AM85}), and thus \Cref{th: sufficient} covers $d$-regular expanders (for $d\le |V(G)|^{1-a}$ for some constant $a>0$). In particular, since a random $d$-regular graph has that $\lambda_2=O(\sqrt{d})$ (see \cite{S23} and references therein), for $d\le |V(G)|^{1-a}$ for some constant $a>0$, Theorem \ref{th: sufficient} covers random $d$-regular graphs.  

Note that some assumption on the `global' edge-expansion similar to \ref{p: global} is necessary, as otherwise the host graph itself could be disconnected (or very weakly connected), and thus there is no hope to show that in the random graph process, \textbf{whp} the hitting times of minimum degree at least $k$ and $k$-connectedness coincide. Our second main result shows that Property \ref{p: local} is also optimal, up to the $O(\log n)$ factor.
\begin{theorem}\label{th: tight}
Let $d,n\in \mathbb{N}$ such that $d=\omega(1)$ and $d\le \sqrt{n}$. Let $\epsilon>0$ be a small constant. There exists a $d$-regular $n$-vertex graph $G$, such that the following holds. For every $U\subseteq V(G)$ with $|U|\le n/2$, $e(U,U^C)\ge \frac{|U|}{10}$, and for every $U\subseteq V(G)$ with $|U|\le \frac{\epsilon d}{3}$, $e(U,U^C)\ge (1-\epsilon)d|U|$. Consider the random graph process on $G$. Let $\tau_1$ be the hitting time of minimum degree at least one, and let $\tau_{conn}$ be the hitting time of connectivity. Then, \textbf{whp}, $\tau_1<\tau_{conn}$.
\end{theorem}
We remark that there are some suitable parity assumptions on $n$ and $d$ which we leave implicit here in order to simplify the statement of \Cref{th: tight}, and that the construction we give to prove Theorem \ref{th: tight} follows closely a construction from \cite{DEKKIso23}. We note that, in fact, we show that in the random subgraph the probability threshold of minimum degree at least one is smaller by a multiplicative constant than the probability threshold of connectivity. Finally, we believe it is interesting to compare Theorem \ref{th: tight} with the results of Gupta, Lee, and Li \cite{GLL21}. Indeed, Theorem \ref{th: tight} shows that assuming $G$ is $d$-edge-connected, for $d\le \sqrt{n}$, does not suffice to guarantee that the probability threshold for the random subgraph to be connected is asymptotically the same as the probability threshold for minimum degree at least one, and indeed Gupta, Lee, and Li \cite{GLL21} considered $d$-regular $d$-edge-connected graphs on $n$ vertices when $d=\omega(\sqrt{n})$.
 
The paper is structured as follows. In Section \ref{s: prelim} we set out some notation and a couple of lemmas which we will use throughout the paper. Section \ref{s: sufficient} is devoted to the proof of Theorem \ref{th: sufficient}, and in Section \ref{s: tight} we prove Theorem \ref{th: tight}. Finally, in Section \ref{s: discussion} we discuss our results and consider avenues for future research.

\section{Preliminaries}\label{s: prelim}
Given a graph $G=(V,E)$ and a subset $S\subseteq V$, we denote by $G[S]$ the induced graph on $S$ and by $N_G(S)$ the external neighbourhood of $S$ in $G$, that is, the set of vertices $v\in V\setminus S$ for which there exists $u\in S$ such that $uv \in E$. We denote by $E_G(S, S^C)$ the set of edges in $G$ with exactly one endpoint in $S$, and let $e_G(S,S^C)=|E_G(S,S^C)|$.
Given $v\in V$, we denote by $deg_G(v,S)$ the number of neighbours of $v$ in $S$.
When the graph in question is obvious, we may omit the subscript. Throughout the paper, all logarithms are with the natural base. We omit rounding signs for the sake of clarity of presentation.

We note that throughout the paper (and, in particular, in the statements of our main results), we make a fairly standard abuse of notation: when considering a $d$-regular $n$-vertex graph $G$, we in fact consider a sequence $(d_k, n_k)_{k\in\mathbb{N}}$ of pairs $(d_k,n_k)\in \mathbb{N}^2$ giving rise to a sequence $(G_k)_{k\in\mathbb{N}}$ of $d_k$-regular $n_k$-vertex graphs. For example, a statement such as `let $d=\omega(1)$, let $n\ge d^2$ [...] there exists a $d$-regular $n$-vertex graph such that [...]' should be interpreted as `for all sequences $(d_k, n_k)_{k\in \mathbb{N}}$ such that for every positive constant $C$, there exists $k_0$ such that for every $k\ge k_0$, and every pair $(d_k,n_k)\in \mathbb{N}^2$ satisfying $d_{k}\ge C$ and $n_k\ge d_k^2$ [...] there exists a sequence $(G_k)_{k\in\mathbb{N}}$ of $d_k$-regular $n_k$-vertex graphs such that[...]'.

We will utilise the following bound on the number of trees on $m$ vertices in a $d$-regular graph that are rooted at a fixed vertex. 
\begin{lemma}[See, for example, Lemma 2 in \cite{BFM98}]\label{l: trees}
Let $G$ be a $d$-regular graph, let $m$ be a positive integer and let $v\in V(G)$. Denote by $t_m(v, G)$ the number of trees on $m$ vertices rooted at $v$ in $G$. Then $t_m(v, G)\le (ed)^{m-1}$.
\end{lemma}

We will also make use of the following lemma, allowing one to find large matchings in percolated subgraphs, which follows immediately from \cite[Lemma 3.8]{DEKK24}.
\begin{lemma}\label{l: matching}
Let $d=\omega(1)$ and let $G$ be a $d$-regular graph. Let $\delta_1>0$ be a constant, let $q=\frac{\delta_1}{d}$, and let $s=\Omega(d)$.
Let $F\subseteq E(G)$ be such that $|F|\ge s$. Then, there exists a constant $\delta_2=\delta_2(\delta_1)>0$ such that $F_{q}$, a random subset of $F$ obtained by retaining each edge independently with probability $q$, contains a matching of size at least $\frac{\delta_2s}{d}$ with probability at least $1-\exp\left\{-\frac{\delta_2s}{d}\right\}$.
\end{lemma}

\section{Proof of Theorem \ref{th: sufficient}}\label{s: sufficient}
Throughout this section let $G=(V, E)$, $d, n, a, k$ be as in Theorem \ref{th: sufficient}. \Cref{th: sufficient} will follow from the next proposition.
\begin{proposition}\label{prop: sufficient}
Let $\phi$ be a function tending arbitrarily slowly to infinity with $n$. Let $p= 1-\left(\frac{\phi}{n}\right)^{1/d}$. Then, \textbf{whp}, $G_p$ has a unique $k$-connected component. All the vertices outside this component, if there are any, have degree at most $k-1$ in $G_p$ and are at distance at least two from each other in $G$. 
\end{proposition}
Observe that when $p=1-\left(\frac{\phi}{n}\right)^{1/d}$, we have that
\begin{align*}
   \Pr{Bin(d,p)\le k-1}=\sum_{i=0}^{k-1}\binom{d}{i}p^i(1-p)^{d-i}=\omega(1/n),
\end{align*}
and thus the expected number of vertices whose degree in $G_p$ is at most $k-1$ tends to infinity. Further, since by \Cref{prop: sufficient} \textbf{whp} every two vertices of degree at most $k-1$ in $G_p$ are at distance at least two from each other in $G$, any edge added in the random graph process that might turn a vertex of degree less than $k-1$ to a vertex of degree $k$ \textbf{whp} is an edge connecting it to the unique $k$-connected component. Then, Theorem \ref{th: sufficient} follows from Proposition \ref{prop: sufficient} using standard methods (see, for example, \cite{B01} and \cite[Lemma 2]{B90}).

Throughout the proof of the proposition, we will often use that
\begin{equation} \label{eq:dp}
    dp=O(\log n).
\end{equation}
To see this, observe that $g(d)=dp=d\left(1-(\frac{\phi}{n})^{1/d}\right)$ is increasing as a function of $d\in [1,n-1]$ and thus the maximum is attained at $d=n-1$. For this value of $d$ we have
\[
p=1-\left( \frac{\phi}{n}\right)^{1/(n-1)}=O\left(\frac{\log n}{n}\right),
\]
which implies the claim.

Before we turn to the proof itself, let us give a broad overview of its strategy. We begin by showing that any two vertices of degree less than $k$ in $G_p$ are at distance at least two in $G$. Then, we utilise a sprinkling/double-exposure argument, similar in spirit to the classical argument of Ajtai, Koml\'os, and Szemer\'edi~\cite{AKS81}. We let $p_2=\frac{1}{d}$ and let $p_1$ be such that $(1-p_1)(1-p_2)=1-p$. Note that then $G_p$ has the same distribution as $G_{p_1}\cup G_{p_2}$. Note that using \eqref{eq:dp} it is clear that $p_1$ is of the form $1-\left(\frac{\phi'}{n}\right)^{1/d}$, for some $\phi'$ which tends to infinity with $n$. Further, $p_1$ satisfies that $dp_1=O(\log n)$. We then turn to show that for any set $K\subseteq V$ of order $k$, there are no components in $G_p[V\setminus K]$ (nor in $G_{p_1}[V\setminus K])$ whose order lies in the interval $[2, Cd\log n]$. Finally, we show that for every set $K\subseteq V$ of order $k$, after sprinkling with probability $p_2$, all components in $G_{p_1}[V\setminus K]$ of order at least $Cd \log n$ merge into a unique component.

To the task at hand, let us begin by showing that vertices of degree less than $k$ in $G_p$ are non-adjacent in $G$.
\begin{lemma} \label{l: distance}
\textbf{Whp}, any two vertices of degree less than $k$ in $G_p$ are at distance at least two in~$G$.
\end{lemma}
\begin{proof}
    Consider any edge $uv$ in $E$.
    The probability that both $u$ and $v$ are vertices of degree less than $k$ in $G_p$ is at most
    \[
    \Pr{Bin(d, p)<k} \Pr{Bin(d-1, p)<k}.
    \]
    
    Let $\mathcal{A}$ be the event that there is an edge of $G$ whose endpoints are vertices of degree less than $k$. Since $G$ is $d$-regular, there are $\frac{dn}{2}$ edges to consider. Then, by the union bound,
    \[
    \Pr{\mathcal{A}} \leq \frac{dn}{2} \Pr{Bin(d, p)<k} \Pr{Bin(d-1, p)<k}.
    \]
    Note that $\Pr{Bin(d, p)<k}=(1+o(1)) \Pr{Bin(d-1, p)<k}$ since $d=\omega(1)$ and $k$ is constant. Thus, 
    \begin{equation*}
        \begin{split}
        \Pr{\mathcal{A}}&\le (1+o(1))\frac{dn}{2}\left(\Pr{Bin(d,p)<k}\right)^2 \\
        &= (1+o(1))\frac{dn}{2}\left(\sum_{i=0}^{k-1}\binom{d}{i}p^i(1-p)^{d-i}\right)^2.
        \end{split}
    \end{equation*}
    Recall that $d \leq n^{1-a}$ for some $a>0$ and $p=1-\left(\frac{\phi}{n}\right)^{1/d}$ for some $\phi$ tending to infinity. Using the bound $\binom{d}{i} \leq d^i$ and \eqref{eq:dp} we obtain
    \begin{equation*}
        \begin{split}
        \Pr{\mathcal{A}}&\le (1+o(1))\frac{dn}{2}\left(O(\log^kn)\cdot (1-p)^{d-k}\right)^2\\
        &=O\left(dn\log^{2k}n\cdot \frac{\phi^{2-2k/d}}{n^{2-2k/d}}\right)\\
        &=O\left(\log^{2k}n\cdot \phi^{2}\cdot n^{-a+o(1)}\right)=o(1). \qedhere
        \end{split}
    \end{equation*}
\end{proof}
We now turn to show a `gap'-like statement in $G_p$. 
\begin{lemma} \label{l: gap}
    \textbf{Whp}, for any set $K\subset V$ with $|K| \leq k$, there are no connected components in $G_p [V\setminus K]$ whose order lies in the interval $[2, Cd\log n]$.
\end{lemma}
\begin{proof}
    For $s \in \mathbb{N}$ denote by $\mathcal{A}_s$ the event that there exists $K\subseteq V$, $|K|\le k$, such that $S$ is a connected component of order $s$ in $G_p[V \setminus K]$.
    If $S$ forms a component in $G_p[V \setminus K]$, then $N_{G_p}(S) \subseteq K$.
    We consider three separate cases, according to the order of $S$.

    First, suppose that $s\in [2,k+1]$. Then, there are two vertices in $S$ which are adjacent in $G_p$ and have all their neighbours in $S\cup K$. That is, there is an edge $uv$ in $G_p$, such that $e_{G_p}(\{u,v\}, V \setminus\{u,v\})\le 2|S\cup K|\le 4k+2$. 
    There are less than $dn$ edges to consider, and thus by the union bound and by \eqref{eq:dp}
    \begin{align*}
        \mathbb{P}\left(\cup_{s=2}^{k+1}\mathcal{A}_s\right)&\le k\cdot dn\cdot\Pr{Bin(2d-1,p)\le 4k+2}\\
        &=kdn\sum_{i=0}^{4k+2}\binom{2d-1}{i}p^{i}(1-p)^{2d-1-i}\\
        &\le kdn\cdot O\left(\log^{4k+2}n(1-p)^{2d-3-4k}\right)\\
        &\le O\left(n^{2-a}\cdot \log^{4k+2}n\cdot \frac{\phi^2}{n^{2-7k/d}}\right)\le \frac{1}{n^{a/2}}=o(1).
    \end{align*}

    Now, assume that $s\in [k+2, \sqrt{d}]$. By \Cref{l: trees} there are at most $n (ed)^{s-1}$ ways to choose a tree of order $s$, and it falls into $G_p$ with probability $p^{s-1}$. Each vertex in $S$ has at most $k$ neighbours in $K$ and at most $s$ neighbours in $S$.
    Therefore, using \eqref{eq:dp}, we get
    \begin{align*}
        \Pr{\mathcal{A}_s}&\le \binom{n}{k}n(ed)^{s-1}p^{s-1}(1-p)^{s(d-k-s)}\\
        &\le O(n^{k+1}e^s\log^sn)\cdot (1-p)^{s(d-2\sqrt{d})}
    \end{align*}
    Recall that $p=1-\left(\frac{\phi}{n}\right)^{1/d}$ for some $\phi$ tending to infinity arbitrary slowly. Using $k+2 \leq s \leq \sqrt{d}$, we can bound
    \begin{align*}
        \Pr{\mathcal{A}_s}&\le O(n^{k+1}e^s\log^sn)\cdot \frac{\phi^s}{n^{s\left(1-\frac{2}{\sqrt{d}}\right)}}\\
        &\le O\left(\frac{(e\cdot \phi \cdot \log n)^s}{n^{s\left(1-\frac{2}{\sqrt{d}}\right)-(k+1)}}\right)\\
        &\le O\left(\frac{(e\cdot \phi \cdot \log n)^{k+2}}{n^{(k+2)\left(1-\frac{2}{\sqrt{d}}\right)-(k+1)}}\right),
    \end{align*}
    By the union bound over the less than $\sqrt{d}$ values of $s$, we obtain that
    \begin{align*}
        \mathbb{P}\left(\bigcup_{s=k+2}^{\sqrt{d}}\mathcal{A}_s\right)&\le \sqrt{d}\cdot O\left(\frac{(e\cdot \phi \cdot \log n)^{k+2}}{n^{(k+2)\left(1-\frac{2}{\sqrt{d}}\right)-(k+1)}}\right)=o(1),
    \end{align*}
    where we used $d\le n^{1-a}$. 

    Finally, we are left with $s\in [\sqrt{d}, Cd\log n]$. By property \ref{p: local}, since $s \le Cd\log n$, we have that $e_G(S,S^C)\ge (1-\epsilon)ds$. At most $kd$ of these edges have an endpoint from $K$, and thus
    \[e_{G[V\setminus K]}(S,S^C) \ge (1-\epsilon)ds - dk.\]
    Once again, by \Cref{l: trees} there are at most $n (ed)^{s-1}$ ways to choose a tree of order $s$, and it falls into $G_p$ with probability at most $p^{s-1}$. Therefore,
    \begin{align*}
        \Pr{\mathcal{A}_s}&\le \binom{n}{k}n(edp)^{s-1}(1-p)^{(1-\epsilon)ds-dk}\\
        &\le n^{k+1}\cdot O\left(\frac{(e\cdot \log n\cdot \phi)^s}{n^{(1-\epsilon)s-k-1}}\right)=o(1/n^2),
    \end{align*}
    where the last equality holds since $s\ge \sqrt{d}=\omega(1)$. Union bound over the less than $Cd \log n<n^2$ possible values of $s$ completes the proof.
\end{proof}

We are now ready to prove Proposition \ref{prop: sufficient}.
\begin{proof}[Proof of Proposition \ref{prop: sufficient}]
Recall that $p_2=\frac{1}{d}$, and that $p_1$ satisfies that $(1-p_1)(1-p_2)=1-p$. In particular, $G_p$ has the same distribution as $G_{p_1}\cup G_{p_2}$, and both Lemma \ref{l: distance} and Lemma \ref{l: gap} hold in $G_{p_1}$ as well. Set $G_1\sim G_{p_1}$ and $G_2\sim G_{p_1}\cup G_{p_2}$.

We begin by exposing $G_1$. Let us first show that after sprinkling with $p_2$, there is a unique $k$-connected component whose order is larger than one. To that end, we will show that for any $K\subseteq V$ of order at most $k$, there is a unique component in $G_{2}[V\setminus K]$ whose order is larger than one. By Lemma \ref{l: gap}, we have that \textbf{whp} there are no components in $G_1[V\setminus K]$ whose order is in the interval $[2, Cd\log n]$, and by Lemma \ref{l: distance}, we have that \textbf{whp} any two vertices whose degree is less than $k$ in $G_1$ are non-adjacent in $G$. We continue assuming that these two properties hold deterministically. 

Fix $K\subseteq V$. Let $W$ be set of vertices in components of order at least $Cd\log n$ in $G_{1}[V\setminus K]$, and let $A\sqcup B$ be a partition of $W$ that respects the components of $G_1[V\setminus K]$, that is, each component of $G_1[V\setminus K]$ is completely contained in either $A$ or $B$. We may assume that $|A|\le |B|$. By the above, for every $v\in V\setminus (K\cup W)$, we have that $N_G(v)\subseteq W\cup K$. In particular, either $deg_G(v,A)\ge \frac{d}{2}-k\ge \frac{d}{3}$ or $deg_G(v,B)\ge \frac{d}{3}$.  

Let $A'$ be $A$ together with all vertices in $v\in V\setminus (K\cup W)$, such that $d(v,A)\ge \frac{d}{3}$. Similarly, let $B'$ be $B$ together with all vertices $v\in V\setminus (K\cup W\cup A')$ satisfying $d(v,B)\ge \frac{d}{3}$. By the above, we have that $V\setminus K=A'\sqcup B'$. 

By Property \ref{p: global}, we have that $e_G(A', V\setminus A')\ge c|A'|\ge c|A|$. Hence, $e_G(A', B')\ge c|A|-kd$. Recalling that $|A|\ge Cd \log n$, we have that $e_G(A', B')\ge c|A|/2$. We now turn to exposing edges with probability $p_2$. By Lemma \ref{l: matching}, there exists some absolute constant $\delta>0$ such that with probability at least $1-\exp\left\{-\frac{\delta c|A|}{2d}\right\}$, there exists a matching $M_1$ in $G_{p_2}$ of order at least $\frac{\delta c|A|}{2d}$ between $A'$ and $B'$. Let $A'_M$ denote the endpoints of the matching $M_1$ in $A'$. 
    
Now, either half of the vertices of $A'_M$ are in $A$, or at least half of $A'_M$ are in $A'\setminus A$. In the first case, with probability at least $1-\exp\left\{-\frac{\delta c|A|}{2d}\right\}$ we have at least $\frac{\delta c|A|}{4d}$ edges in $G_{p_2}$ between $B'$ and $A$. Suppose now that we are in the second case. Each of the vertices in $A'_M\setminus A$ has at least $\frac{d}{3}$ neighbours in $A$, and we thus have a set of $|A'_M\setminus A|d/3\ge |A'_M|d/6$ edges between $A'_M$ and $A$. By Lemma \ref{l: matching}, with probability at least $1-\exp\left\{-\frac{\delta |A'_M|}{6}\right\}$ there is a matching $M_2$ between $A'_M$ and $A$ of size at least $\delta|A'_M|/6$.
    
Thus, in both cases, with probability at least $1-\exp\left\{-\frac{\delta^2c|A|}{13d}\right\}$ there are at least $\frac{\delta^2 c|A|}{12d}$ paths of length at most two in $G_{p_2}$ between $B'$ and $A$.
Let $B'_P$ denote the endpoints of these paths in $B'$. Each vertex in $B'_P$ has at least $\frac{d}{3}$ edges to $B$ (see figure \ref{fig:proof notation}). Then, the probability that there is no path between $A$ and $B$ in $G_{p_2}$ is at most \[\exp\left\{-\frac{\delta^2c|A|}{13d}\right\}+(1-p_2)^{\frac{d}{3}|B'_P|}\le \exp\left\{-\frac{\delta^2c|A|}{40d}\right\}.\] Hence, by the union bound, the probability that there is a partition $A\sqcup B$ of $W$ with no paths between $A$ and $B$ in $G_{p_2}[V\setminus K]$ is at most
\begin{align*}
    \sum_{|A|=Cd\log n}^{n}\binom{n/(Cd\log n)}{|A|/(Cd\log n)}\exp\left\{-\frac{\delta^2c|A|}{40d}\right\}&\le \sum_{|A|=Cd\log n}^n\exp\left\{\frac{|A|}{d}\left(\frac{1}{C}-\frac{\delta^2c}{40}\right)\right\}\\
    &\le \sum_{|A|=Cd\log n}^n\exp\left\{-\frac{\delta^2c|A|}{80d}\right\}=o(1/n^k),
\end{align*}
where we assumed that $C$ is sufficiently large with respect to $c$ and $k$. 

Thus, by the union bound over the less than $n^k$ possible choices of $K$, \textbf{whp}, for any choice of $K$ all components of $G_2[V\setminus K]$ of order at least $C d \log n$ merge when we sprinkle with $p_2$, that is, \textbf{whp} in $G_2$ there is a unique $k$-connected component of order at least two and all the other components are vertices of degree less than $k$, which are \textbf{whp} non-adjacent.
\end{proof}

\begin{figure}
    \centering
\begin{tikzpicture}
    \draw[] (0,0) ellipse (4cm and 2cm);
    \node at (-2,0) {$A$};
    \node at (2,0) {$B$};
    \draw[] (4,-2) ellipse (1cm and 0.5cm);
    \node at (4,-2) {$K$};
    \path[draw=black] (0.3,2) -- (0.3,-2);

    \draw[fill=gray!20] (2,1) ellipse (0.5cm and 0.2cm);
    \path[draw=black] (2.3,2) -- (1.48,1.05);
    \path[draw=black] (2.3,2) -- (2.48,1);

    \path[draw=blue, dashed] (-1.3,2.3) -- (0.8,1.5);
    \path[draw=blue] (-1.2,1) -- (0.8,1);
    \path[draw=blue] (-1.2,0.5) -- (0.8,0.5);
    \path[draw=blue, dashed] (-1.9,2.8) -- (2.3,2);
    \draw [blue, loosely dotted] (0,0.2) -- (0,-0.3);
    \node[blue] at (-0.4,0) {$M_1$};

    \path[draw=red, dashed] (-1.3,2.3) -- (-1.5,0.7);
    \path[draw=red, dashed] (-1.9,2.8) -- (-2.15,0.7);
    \path[draw=red] (-2.5,1.9) -- (-2.65,0.7);
    \node[red] at (-3,0.9) {$M_2$};

    \node (W) at (-3.3,2.2) [circle,draw, fill, scale=0.6] {};
    \node (W) at (-1.3,2.3) [circle,draw=blue, fill=blue, scale=0.6] {};
    \node (W) at (-2.5,1.9) [circle,draw, fill, scale=0.6] {};
    \node (W) at (-1.9,2.8) [circle,draw=blue, fill=blue, scale=0.6] {};

    \node (W) at (-1.2,1) [circle,draw=blue, fill=blue, scale=0.6] {};
    \node (W) at (-1.2,0.5) [circle,draw=blue, fill=blue, scale=0.6] {};

    \node (W) at (0.8,1.5) [circle,draw=gray, fill=gray, scale=0.6] {};
    \node (W) at (0.8,1) [circle,draw=gray, fill=gray, scale=0.6] {};
    \node (W) at (0.8,0.5) [circle,draw=gray, fill=gray, scale=0.6] {};

    \node (W) at (2.3,2) [circle,draw=gray, fill=gray, scale=0.6] {};
    \node (W) at (1,2.8) [circle,draw, fill, scale=0.6] {};
    \node (W) at (3.5,2.3) [circle,draw, fill, scale=0.6] {};

    \path[draw=black, rounded corners] (0.3,2) -- (0.3,3.2) -- (-4,3.2) -- (-4,0);
    \path[draw=black, rounded corners] (0.3,2) -- (0.3,3.2) -- (4,3.2) -- (4,0);

    \node at (-2,3.5) {$A'$};
    \node at (2,3.5) {$B'$};
\end{tikzpicture}
\caption{The matching $M_1$ (blue) in $G_{p_2}$ connects the sets $A'$ and $B'$. If many vertices in $A'_M$ (blue) lie in $A' \setminus A$, there is a matching $M_2$ in $G_{p_2}$ (red) between $A$ and $A' \setminus A$ whose edges together with those of $M_1$ form many paths of length two in $G_{p_2}$ (dashed) between $A$ and $B'$. The endpoints of these paths lie in $B'_P$ (gray). Note that each vertex in $B' \setminus B$ has $\Omega(d)$ neighbours in $B$ (and likewise for vertices in $A' \setminus A$ and $A$).}
    \label{fig:proof notation}
\end{figure}
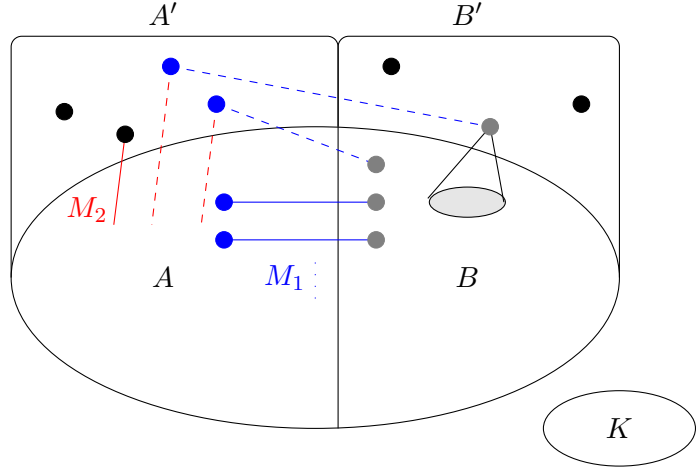

\section{Proof of Theorem \ref{th: tight}}\label{s: tight}

The proof follows a construction from \cite{DEKKIso23}.
\begin{proof}[Proof of \Cref{th: tight}]
Let $\epsilon >0$ and $d=\omega(1)$ be given.
Choose $n\ge d^2$ such that $n$ is divisible by $d+2$. Set $d_1=\frac{10}{19}\cdot d$. We construct the graph as follows.

Let $H$ be a $(\frac{n}{d+2}, d_1, \lambda)$-graph with $\lambda=O(\sqrt{d_1})$ (a typical random $d_1$-regular graph satisfies this, see for example \cite{S23} and reference therein).
Furthermore, for each $v \in V(H)$ let $K(v)$ be a clique on $d+1$ vertices and let $M(v) \subseteq E(K(v))$ be a matching of size $\frac{d-d_1}{2}$ in $K(v)$. Set $F(v)$ to be the graph whose vertex set is $V(K(v))$, and whose edge set is $E(K(v))\setminus M(v)$.
Let $G$ be the graph whose vertex set is $V(H) \cup \bigcup_{v \in V(H)} V(F(v))$, and whose edge set is $E(H)\cup\bigcup_{v\in V(H)}E(F(v))$ together with the edges from every $v \in V(H)$ to the endpoints of $M(v)$.

Note, that $G$ is a $d$-regular graph on $n$ vertices.
Set $c_1=1/10$ and $c_2=\epsilon/3$.
By construction, $H$ is an expander (in fact, a pseudo-random graph) and thus by the expander mixing lemma~\cite{AC88}, for $U \subseteq V(H)$ with $|U|\leq \frac{n}{2(d+2)}$, 
\begin{equation} \label{expansionH global}
    e_H(U, U^C) \geq \frac{d_1}{3} |U|.
\end{equation}
Moreover, since $d\le \sqrt{n}$ we have that $c_2d\le \frac{\epsilon n}{d+2}=\epsilon|V(H)|$. Thus, again by the expander mixing lemma, for $U\subseteq V(H)$ with $|U|\leq c_2 d$,
\begin{equation} \label{expansionH local}
    e_H(U, U^C) \geq (1-\epsilon) d_1 |U|.
\end{equation}
We use the following claim from \cite{DEKKIso23} about the expansion of $G$.

\begin{claim}[Claim 5.1 in \cite{DEKKIso23}]
For $U \subseteq V(G)$ with $|U| \leq n/2$,
\[
e(U, U^C) \geq c_1|U|=  \frac{|U|}{10}.
\]
\end{claim}

Furthermore, we claim that small sets expand even better.
\begin{claim}
For $U \subseteq V(G)$ with $|U| \leq c_2 d$,
\[
e(U, U^C) \geq (1- \epsilon) d|U|.
\]
\end{claim}
\begin{proof}
    Let $U_H=U \cap V(H)$ and $U_{F(v)}= U \cap V(F(v))$ and note that $U= U_H \sqcup \bigsqcup_{v \in V(H)} U_{F(v)}$ is a partition of $U$.
    
   By \eqref{expansionH local}, since $|U_H|\le |U| \leq c_2 d$,
    \[
    e_H(U_H, U_H^C) \geq (1-\epsilon) d_1 |U_H|.
    \]
    For every $v\in V(H)$, each $v \in U_H$ sends at least $d-d_1-|U_{F(v)}|$ edges to $F(v) \setminus U_{F(v)}$. Further, for every $v\in V(H)$, each vertex in $U_{F(v)}$ sends at least $d-|U_{F(v)}|-1$ edges to vertices in $F(v) \setminus U_{F(v)}$. Hence, recalling that $|U_{F(v)}| \leq c_2 d = \epsilon d/3$, we get
    \[
    e_{F(v)}(U_{F(v)}, U_{F(v)}^C) \geq |U_{F(v)}| (d-|U_{F(v)}|-1) \geq (1-\epsilon) d |U_{F(v)}|.
    \]
   Altogether we obtain,
    \begin{align*}
        e(U, U^C) &\geq e_H\left(U_H, U_H^C\right) + \sum_{U_{F(v)} \neq \varnothing} e_{F(v)}\left(U_{F(v)}, U_{F(v)}^C\right) + e_G\Bigl(U_H, \bigcup_{v \in U_H} (F(v) \setminus U_{F(v)})\Bigr)\\
        & \geq |U_H|(1-\epsilon) d_1 + \sum_{U_{F(v)} \neq \varnothing}(1-\epsilon) d |U_{F(v)}| + \sum_{v \in U_H} (d-d_1-|U_{F(v)}|)\\
        & \geq |U_H|(1-\epsilon) d_1 + |U\setminus U_H| (1- \epsilon) d + |U_H| (1-\epsilon)(d - d_1)\\
        & \geq |U| (1-\epsilon) d. \qedhere
    \end{align*}
\end{proof}

It remains to show that, \textbf{whp}, $\tau_1 < \tau_{conn}$ (in fact, we will show that the probability thresholds are different). Let $p \in (0,1)$ such that
\[
(1-p)^d = \frac{1}{n \log d}.
\]
Then, in $G_p$ the expected number of isolated vertices is
\[
n (1-p)^d = \frac{1}{\log d}=o(1),
\]
and thus \textbf{whp} the minimum degree in $G_p$ is at least one.

Now, let $X$ be the number of $v \in V(H)$ such that $F(v)$ is disconnected from $H$ in $G_p$. Note that each $F(v)$ is connected to $H$ via $d-d_1$ edges in $G$, and that for each $v \in V(H)$, the events that $F(v)$ is disconnected from $H$ are independent. Thus,
\begin{align*}
    \Pr{X= 0}&= \left(1-(1-p)^{d-d_1}\right)^{n/(d+2)}\\
    &\le \exp\left\{-\left(\frac{1}{n \log d}\right)^{9/19}\cdot \frac{n}{d+2}\right\}=o(1),
\end{align*}
where in the last equality we used $d\le \sqrt{n}$. Thus, \textbf{whp}, there is at least one $v \in V(H)$ such that $F(v)$ is disconnected from $H$ in $G_p$, and therefore \textbf{whp} $G_p$ is disconnected.
\end{proof}

\section{Discussion}\label{s: discussion}
We proved sufficient conditions on a regular graph $G$, guaranteeing that in the random graph process on $G$ the hitting times of minimum degree $k$ and $k$-connectedness coincide (for any constant $k$). In particular, these conditions cover pseudo-random $(n,d,\lambda)$-graphs and regular Cartesian product graphs with base graphs of bounded order (such as the $d$-dimensional binary hypercube), and confirm a conjecture of Joos \cite{J15}. We further demonstrated the tightness of our assumptions, in the sense that there are $d$-regular graphs where Property \ref{p: local} holds in a weaker form, and typically the hitting time of $k$-connectedness does not coincide with that of minimum degree $k$ (in fact, the probability threshold is different). It would be interesting to determine how tight our assumptions are: can one remove the $O(\log n)$ factor in \ref{p: local}, and thus match the local edge-expansion as in the construction given in Section \ref{s: tight}? 

The broad strategy used here can also be found in several proofs addressing the uniqueness of the giant component in the supercritical regime, see \cite{AKS81,CDE24,DEKK22, L22}. Let us stress again that all that is needed for the proof of our result is some knowledge of the edge isoperimetric profile of the graph. Indeed, as noted in the introduction, Properties \ref{p: global} and \ref{p: local} from \Cref{th: sufficient} resemble similar assumptions in \cite{DK2401,DK2402}, under which the uniqueness of the giant component in a general setting was shown (and which were shown to be essentially tight).

Further, consider an $n$-vertex graph $G$ with minimum degree $d$ and maximum degree $\Delta=O(d)$, with $d\le n^{1-a}$ for some constant $a>0$. If $G$ has Property \ref{p: global}, and satisfies that for every $U\subseteq V(G), |U|\le Cd\log n$ we have that $|E(G[U])|\le \epsilon d |U|$ for large enough $C$ (with respect to $c,k$ and $\Delta/d$), it follows from the above proof of Theorem \ref{th: sufficient} that \textbf{whp} the hitting time of minimum degree at least $k$ and of $k$-connectedness are the same. It would be interesting to see whether the same holds without the assumption that $\Delta=O(d)$, that is, with only the assumption of minimum degree.

In \cite{DG24}, the authors showed that for regular Cartesian product graphs (where the base graphs are connected and of bounded order), the hitting times of minimum degree one, connectedness and the existence of a nearly perfect matching (that is, a matching missing at most one vertex) coincide \textbf{whp}, generalising a result of Bollob\'as \cite{B90} on the hypercube. The existence of a perfect matching was shown using the product structure of the graph quite explicitly. Hence, the following is quite natural to ask.
\begin{question}
Given $k \in \mathbb{N}$, what minimal assumptions on $G$ suffice to guarantee that in the random graph process on $G$ the hitting times of minimum degree $k$ and the existence of $k$ disjoint perfect matchings coincide \textbf{whp}?
\end{question}

\subsection*{Acknowledgements}
The authors would like to thank Yoshiharu Kohayakawa for proposing several questions regarding the results from \cite{DG24} that led to this project. 
The authors would like to express their sincere gratitude to their supervisors Joshua Erde, Mihyun Kang, and Michael Krivelevich for their invaluable guidance, support, and encouragement throughout the course of this research.
The second author was supported in part by the Austrian Science Fund (FWF) [10.55776/DOC183].

\bibliographystyle{abbrv} 
\bibliography{perc} 

\end{document}